\documentclass{amsart}

\usepackage{graphicx}
\usepackage{amsfonts}
\usepackage{amssymb}
\usepackage{amsmath}
\usepackage{xcolor} 

\usepackage[hidelinks]{hyperref} 
\hypersetup{
    colorlinks=true,
    linkcolor=red,
    citecolor=blue,
    filecolor=red,
   urlcolor=red}

\newtheorem{theorem}{Theorem}

\newtheorem{lemma}{Lemma}

\newtheorem{corollary}{Corollary}
\newtheorem{fact}{Fact}

\newcommand{\GG}{\mathbf{G}}
\newcommand{\cc}{\mathfrak{c}}
\newcommand{\hcc}{\hat{\mathfrak{c}}}
\newcommand{\tcc}{\tilde{\mathfrak{c}}}
\newcommand{\node}[3]{\mathbf{[\boldsymbol{#3};\boldsymbol{#1},\boldsymbol{#2}]}}

\begin{document}

\title{On the closed Ramsey numbers $R^{cl}(\omega+n,3)$}
\author{Burak Kaya}
\address{Department of Mathematics, Middle East Technical University, 06800, \c{C}ankaya, Ankara, Turkey\\
}
\email{burakk@metu.edu.tr}

\author{Irmak Sa\u{g}lam}
\address{Department of Mathematics, Middle East Technical University, 06800, \c{C}ankaya, Ankara, Turkey\\
}
\email{saglam.irmak@metu.edu.tr}


\begin{abstract} In this paper, we contribute to the study of topological partition relations for pairs of countable ordinals and prove that, for all integers $n \geq 3$,
\begin{align*}
R^{cl}(\omega+n,3) &\geq \omega^2 \cdot n + \omega \cdot (R(n,3)-n)+n\\
R^{cl}(\omega+n,3) &\leq \omega^2 \cdot n + \omega \cdot (R(2n-3,3)+1)+1
\end{align*}
where $R^{cl}(\cdot,\cdot)$ and $R(\cdot,\cdot)$ denote the closed Ramsey numbers and the classical Ramsey numbers respectively. We also establish the following asymptotically weaker upper bound
\[ R^{cl}(\omega+n,3) \leq \omega^2 \cdot n + \omega \cdot (n^2-4)+1\]
eliminating the use of Ramsey numbers. These results improve the previously known upper and lower bounds.
\end{abstract}

\maketitle

\tableofcontents

\section{Introduction}

Partition relations for cardinals and ordinals were first introduced and studied by Erd\H os and Rado in \cite{ErdosRado53} and \cite{ErdosRado56}. These notions were later generalized to topological spaces by Baumgartner in \cite{Baumgartner86}. Recently, Caicedo and Hilton continued this study and provided upper bounds for topological and closed Ramsey numbers for various pairs of countable ordinals in \cite{CaicedoHilton17}.

In this paper, we shall improve these bounds for the closed Ramsey numbers $R^{cl}(\omega+n,3)$. Before we state our main results, let us recall some basic definitions and results.

For a set $X$ and $k \in \omega$, we set $[X]^k=\{Y \subseteq X: |Y|=k\}$. Given ordinals $\alpha, \beta$ and $X \subseteq \alpha$, we say that $X$ is \textit{order-homeomorphic} to $\beta$ if there exists an order-isomorphism $f: X \rightarrow \beta$ which is also a homeomorphism (with respect to the order topologies.)

Let $\alpha$ and $\beta$ be ordinals. For an ordinal $\gamma$, one writes $\gamma \rightarrow (\alpha,\beta)^2$ if for every function $\cc: [\gamma]^2 \rightarrow \{0,1\}$, there exists a subset $X \subseteq \gamma$ such that
\begin{itemize}
\item $[X]^2 \subseteq \cc^{-1}(0)$ and $X$ is order-homeomorphic to $\alpha$, or
\item $[X]^2 \subseteq \cc^{-1}(1)$ and $X$ is order-homeomorphic to $\beta$.
\end{itemize}
The closed Ramsey number $R^{cl}(\alpha,\beta)$ is the least ordinal $\gamma$ such that $\gamma \rightarrow (\alpha,\beta)^2$ should such an ordinal exist at all.

For example, we have $R^{cl}(\omega+1,3)=\omega^2+1$ and $R^{cl}(\omega+2,3)=\omega^2 \cdot 2 + \omega + 2$. For details, see \cite[Theorem 4.1, Lemma 5.2 and Lemma 5.3]{CaicedoHilton17}. The following was also implicitly proven in \cite[Corollary 5.6]{CaicedoHilton17}.
\begin{fact}\cite{CaicedoHilton17}\label{upperbound} $R^{cl}(\omega+n,3) \leq \omega^2 \cdot (R(n-1,3)+1) + \omega \cdot (n-1) + n$ for every positive integer $n \geq 3$, where $R(\cdot,\cdot)$ denotes the classical Ramsey number.
\end{fact}
On the other hand, no non-trivial lower bounds have been given for $R^{cl}(\omega+n,3)$. See the authors' own comments following the proof of \cite[Corollary 5.6]{CaicedoHilton17}. Our first main result is the following.
\begin{theorem}\label{mainresult-lower} For every positive integer $n \geq 3$, we have
\[ R^{cl}(\omega+n,3) \geq \omega^2 \cdot n + \omega \cdot (R(n,3)-n)+n\]
\end{theorem}
Our second main result is a strengthening of Fact \ref{upperbound}. More precisely, we prove the following theorem, which, together with Theorem \ref{mainresult-lower}, shows that the correct coefficient of $\omega^2$ in $R^{cl}(\omega+n,3)$ is $n$.

\begin{theorem}\label{mainresult-upper} For every positive integer $n \geq 3$, we have
\[ R^{cl}(\omega+n,3) \leq \omega^2 \cdot n + \omega \cdot (R(2n-3,3)+1)+1\]
\end{theorem}

It is well-known that the Ramsey numbers $R(n,3)$ have asymptotic order of magnitude $n^2 / \ln(n)$, see \cite{Kim95}. On the other hand, the exact computation of $R(n,3)$ is a notorious combinatorial problem. For this reason, we also prove the following (asymptotically weaker) result to get rid of these Ramsey numbers.

\begin{theorem}\label{mainresult-upper-2} For every positive integer $n \geq 3$, we have
\[ R^{cl}(\omega+n,3) \leq \omega^2 \cdot n + \omega \cdot (n^2-4)+1\]

\end{theorem}

Even though this upper bound is asymptotically worse than that of Theorem \ref{mainresult-upper}, one can check that it is indeed better for small $n$ values, at least, for $3 \leq n \leq 7$. This paper is organized as follows.

In Section 2, we shall briefly recall the basic definitions and notions that are used in the proofs, most of which appeared in \cite{Mermelstein19}. For the self-containment of this paper, we briefly include this background material.

In Section 3, we will prove Theorem \ref{mainresult-lower} by constructing a special triangle-free graph on a partition of $\omega^2 \cdot n + \omega \cdot (R(n,3)-n)+(n-1)$ that induces a coloring witnessing this lower bound.

In Section 4, we shall first prove a sequence of technical lemmas regarding special types of colorings that are introduced in Section 2. We will then prove Theorem \ref{mainresult-upper}.

In Section 5, using the ideas that are employed in Section 4, we will prove Theorem \ref{mainresult-upper-2} by an argument which is a variation of the proof of Theorem \ref{mainresult-upper}.

\textbf{Acknowledgements.} This paper is a part of the second author's master's thesis written under the supervision of the first author at the Middle East Technical University. The authors would like to thank Omer Mermelstein for his comments on an early draft of this paper as well as his clarifications regarding canonical colorings.

\section{Preliminaries}

\subsection{Basic terminology and definitions} In this subsection, we shall recall some basic terminology and definitions that are used throughout this paper.

Let $\gamma$ be an ordinal. A function $\cc: [\gamma]^2 \rightarrow \{0,1\}$ is called a \textit{coloring} of $\gamma$ with two colors. A subset $X \subseteq \gamma$ is said to be \textit{homogeneous of color $i$} if we have $[X]^2 \subseteq \cc^{-1}(i)$. For simplicity, we shall say that $X$ is a \textit{red (respectively, blue) homogeneous closed copy of $\alpha$} if $X$ is order-homeomorphic to $\alpha$ and is homogeneous of color $0$ (respectively, $1$.)

It is well-known that every non-zero ordinal $\gamma$ can uniquely be written as $$\gamma=\omega^{\beta_1}+\omega^{\beta_2}+\dots+\omega^{\beta_n}$$ where $\beta_1 \geq \beta_2 \geq \dots \geq \beta_n$ are ordinals. For every ordinal $\alpha \leq \gamma$, we set
$$CNF_{\gamma}(\alpha)=\min\{1 \leq k \leq n\ |\ \alpha \leq \omega^{\beta_1}+\omega^{\beta_2}+\dots+\omega^{\beta_k}\}$$
By regrouping the terms together, we can also uniquely express $\gamma$ as
$$ \gamma=\omega^{\alpha_1} \cdot k_1 + \omega^{\alpha_2} \cdot k_2 + \dots + \omega^{\alpha_m} \cdot k_m$$
where $\alpha_1 > \alpha_2 > \dots > \alpha_m$ are ordinals and $k_1,k_2,\dots,k_m < \omega$. In this case, we define \textit{the Cantor-Bendixson rank of $\gamma$} as the ordinal $CB(\gamma)=\alpha_m$ and define $L(\gamma)=k_m$. To avoid trivialities, we also define $CB(0)=0$ and $L(0)=1$.

Next will be defined an ordering on ordinals which first appeared in \cite{CaicedoHilton17}. Consider the relation $<^*$ on the class of ordinals given by
\[ \beta <^* \alpha \text{ if and only if } \alpha=\beta+\omega^{\theta} \text{ for some } \theta>CB(\beta)\]
for all ordinals $\alpha,\beta$. We will also write $\beta \triangleleft^* \alpha$ if $\alpha$ is the unique immediate successor of $\beta$ with respect to the relation $<^*$. For later use, we define the sets
\[T(\alpha)=\{\beta: \beta <^* \alpha\} \cup \{\alpha\}\]
\[T^{=k}(\alpha)=\{\beta \in T(\alpha): CB(\beta)=k\}\]
For a graphical representation of the relation $<^*$ on the ordinal $\omega^2 \cdot n + \omega \cdot K+1$ as a forest, see Figure \ref{figurelargesets}.

Before we define special types of colorings, we need to recall the following definition that first appeared in \cite[Section 2]{Mermelstein19}. A \textit{skeleton} of an ordinal $\gamma < \omega^{\omega}$ is a subset $I \subseteq \gamma$ such that
\begin{itemize}
\item $I$ is order-homeomorphic to $\gamma$ and
\item For all $x,y \in I$, $x <^* y$ if and only if $\rho(x) <^* \rho(y)$, where $\rho: I \rightarrow \gamma$ is the unique order-homeomorphism.
\end{itemize}

For two sets of ordinals $I \subseteq J$, we say that $I$ is a skeleton of $J$ if $\rho[I]$ is a skeleton of $ord(J)$, where $ord(J)$ denotes the order-type of $J$ and $\rho: J \rightarrow ord(J)$ is the unique order-preserving bijection. Given a skeleton $I \subseteq \gamma$ and a coloring $\cc: [\gamma]^2 \rightarrow \{0,1\}$, we define \textit{the induced coloring of $\cc$ with respect to $I$} as the coloring $\cc_{I}: [\gamma]^2 \rightarrow \{0,1\}$ given by $$\cc_{I}(\{\alpha,\beta\})=\cc(\{f(\alpha),f(\beta)\})$$
where $f: \gamma \rightarrow I$ is the unique order-homeomorphism. It is straightforward to check that the image of any homogeneous closed copy of $\theta$ in $\gamma$ with respect to $\cc_{I}$ under the map $f$ is a homogeneous closed copy of $\theta$ in $I \subseteq  \gamma$ with respect to $\cc$. We shall later use this observation to assume without loss of generality that our colorings have special properties.

\newpage

\begin{figure}
  \includegraphics[width=\textwidth,height=\textheight,keepaspectratio]{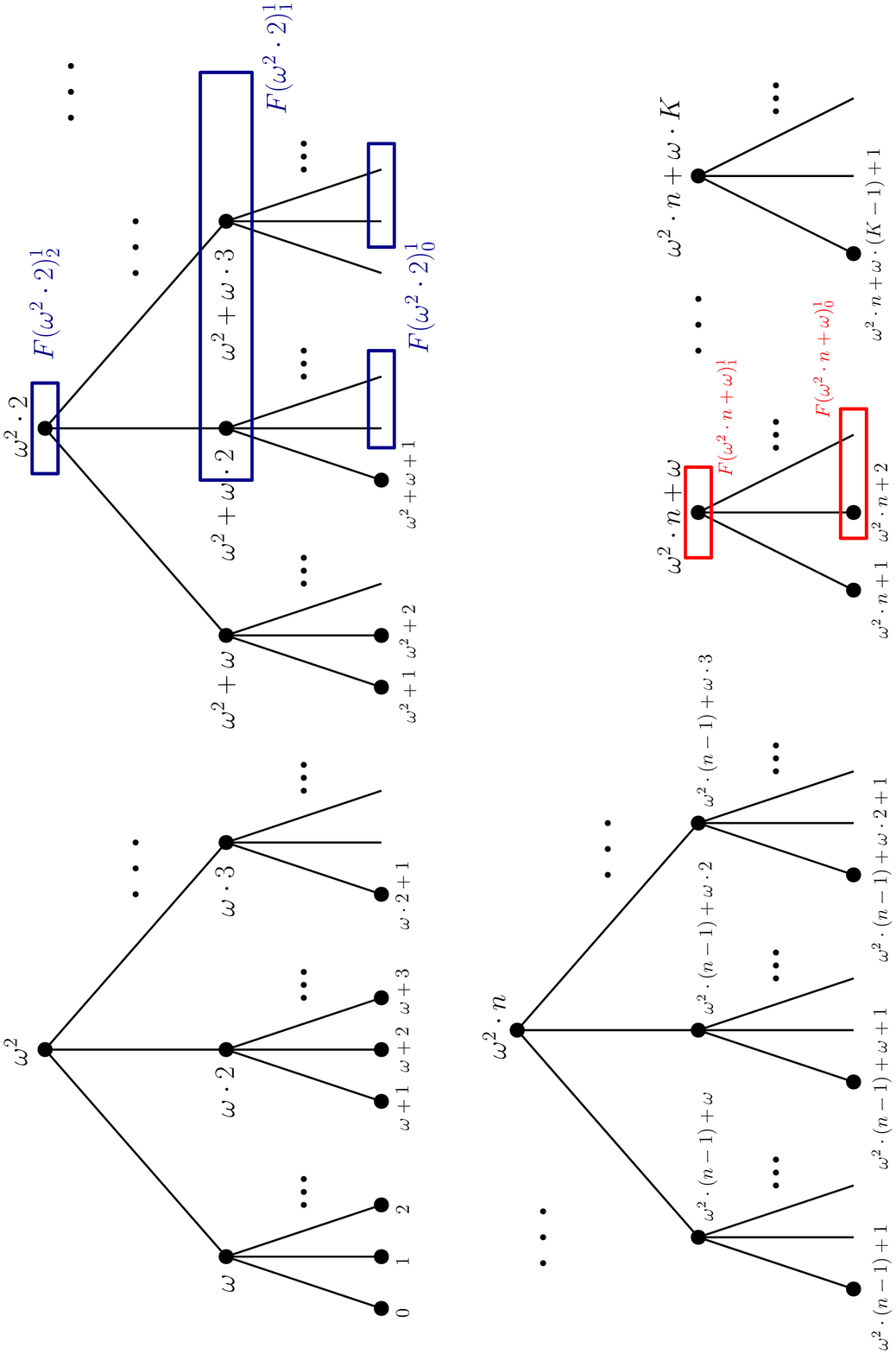}
  \caption{A representation of $<^*$ on the ordinal $\omega^2 \cdot n + \omega \cdot K +1$.}
  \label{figurelargesets}
\end{figure}

\newpage

Let $k,r \in \mathbb{N}$. For each $0 \leq m \leq k$, the sets $F(\omega^k)^r_m$ are defined recursively as follows.
\[F(\omega^k)^r_m=\begin{cases} \{ \omega^k\} & \text{ if } m=k,\\ 
\bigcup_{\alpha \in F(\omega^k)^r_{m+1}} \{ \beta \in \omega^k :\ \beta \triangleleft^* \alpha,\ \ L(\beta)>r  \}& \text{ if } m< k. \end{cases} \]
This definition is extended to all ordinals less than $\omega^\omega$ as follows. For every $\theta<\omega^\omega$ with $CB(\theta) = 0$, we set $F(\theta)^r_0=\{\theta\}$ and for every $\theta<\omega^\omega$ with $CB(\theta) \neq 0$, we define
\[ F(\theta)^r_m=\rho^{-1}\left[F\left(\omega^{CB(\theta)}\right)^r_m\right]\]
where $\rho: \{\alpha: \alpha <^* \theta\} \cup\{\theta\} \rightarrow \omega^{CB(\theta)}+1$ is the unique order preserving map. An illustration of these sets inside the ordinal $\omega^2 \cdot n + \omega \cdot K + 1$ is given in Figure \ref{figurelargesets}.

\subsection{Special types of colorings} In this subsection, we shall define the notions of an $\omega$-homogeneous coloring, a normal coloring and a canonical coloring. The latter two notions first appeared in \cite{Mermelstein19}.

A coloring $\mathfrak{c}: [\gamma]^2 \rightarrow \{0,1\}$ is said to be \textit{$\omega$-homogeneous} if for all $\alpha < \gamma$ there exists $c_{\alpha}\in\{0,1\}$ such that $\mathfrak{c}(\{\beta,\theta\})=c_{\alpha}$ for all $\beta,\theta \triangleleft^* \alpha$. In other words, an $\omega$-homogeneous coloring is a coloring for which the children of each node in the tree representation of $\gamma$ with respect to $<^*$ form a homogeneous copy of $\omega$.

\begin{lemma}\label{omegahomcolor} Let $\gamma < \omega^{\omega}$ be an ordinal. For every coloring $\mathfrak{c}: [\gamma]^2 \rightarrow \{0,1\}$ there exists a skeleton $I \subseteq \gamma$ such that $\cc_{I}: [\gamma]^2 \rightarrow \{0,1\}$ is $\omega$-homogeneous.
\end{lemma}
\begin{proof} We will first prove the result for ordinals of the form $\omega^k+1$.
Let $k \in \mathbb{N}$ and let $\mathfrak{c}: [\omega^k+1]^2 \rightarrow \{0,1\}$ be a coloring. For every $\delta \leq \omega^{k}$, we define $H({\delta}) \subseteq \delta+1$ inductively on the Cantor-Bendixson rank of $\delta$ as follows.
\begin{itemize}
\item If $CB(\delta)=0$, then we set $H({\delta})=\{\delta\}$.
\item If $CB(\delta)>0$, then choose some infinite homogeneous $J_{\delta} \subseteq \{\alpha: \alpha \triangleleft^* \delta\}$. Observe that such a set $J_{\delta}$ must exist by the infinite Ramsey theorem. Now set $H({\delta})=\{\delta\} \cup \bigcup_{\alpha \in J_\delta} H({\alpha})$.
\end{itemize}
A straightforward induction on $1 \leq i \leq k$ implies that, for all $\lambda \leq \omega^k$ such that $CB(\lambda)=i$, the set $H(\lambda)$ is a skeleton of $$\{\theta: \theta <^* \lambda \} \cup\{\lambda\}$$ Consequently, $I=H(\omega^k)$ is a skeleton of $\omega^k+1$. That $\cc_I$ is $\omega$-homogeneous trivially follows from the choice of $J_{\delta}$'s.

To finish the proof, let $\gamma < \omega^{\omega}$ be an ordinal. The claim clearly holds for $\gamma=0$. So suppose that $\gamma=\omega^{m_1}+\dots+\omega^{m_n}$ where $m_1 \geq \dots \geq m_n$ are natural numbers. Let $\mathfrak{c}: [\gamma]^2 \rightarrow \{0,1\}$ be a coloring. For each $1 \leq i \leq n$, consider the set
$$\{\omega^{m_1}+\dots+\omega^{m_{i-1}}+\alpha: \alpha \leq \omega^{m_i}\}$$
which is a copy of $\omega^{m_i}+1$. Let $I_i$ be a skeleton obtained by applying the argument above with the restriction of $\cc$ to this copy. Then it is easily verified that
\[ I=\left(\bigcup_{i=1}^{n} I_i\right)-\{\gamma\}\]
is a skeleton for which $\cc_I$ is $\omega$-homogeneous.
\end{proof}

We should mention that, our use of $\omega$-homogeneous colorings in the proofs is non-essential and is due to our not wanting to apply the infinite Ramsey theorem repeatedly.

We now recall the definition of a normal coloring. A coloring $\cc:[\gamma]^2 \rightarrow \{0,1\}$ is said to be \textit{normal} if for all $\beta_1 <^* \beta_2<\gamma$, the color $\mathfrak{c}(\{\beta_1,\beta_2\})$ solely depends on $CB(\beta_1)$, $CB(\beta_2)$ and $CNF_\gamma(\beta_2)$, that is, there is a function $\hcc$ independent of $\beta_1$ and $\beta_2$ such that, for all $\beta_1 <^* \beta_2<\gamma$, we have
\[
\mathfrak{c}(\{\beta_1, \beta_2\}) = \hcc(CNF_\gamma(\beta_2), CB(\beta_2), CB(\beta_1))
\]
In other words, within each connected component of the tree representation of $\gamma$ with respect to $<^*$, the color of a pair consisting of $<^*$-related elements depends only the levels of the nodes.

Next will be defined the notion of a canonical coloring. For our purposes, we shall only restrict our attention to successor ordinals. Suppose that $$\gamma=\omega^{m_1}+\omega^{m_2}+\dots+\omega^{m_n}+1$$ where $m_1 \geq m_2 \geq \dots \geq m_n$ are natural numbers. A coloring $\cc:[\gamma]^2 \rightarrow \{0,1\}$ is said to be \textit{canonical} if the following conditions are satisfied.
\begin{itemize}
\item[a.] $\cc$ is normal,
\item[b.] For all $\alpha < \gamma$, there exists $r \in \mathbb{N}$ such that for all $\theta<\gamma$ and ${\ell} \leq CB(\theta)$ there is a color $c_{\alpha}(\theta,{\ell}) \in \{0,1\}$ with
\[ \{\beta \in T^{={\ell}}(\theta):\ \mathfrak{c}(\{ \alpha, \beta \}) =c_{\alpha}(\theta,{\ell}) \} \supseteq F(\theta)^r_{\ell}\]
\item[c.] For all $\alpha,\beta<\gamma$ with $CNF_{\gamma}(\alpha)=CNF_{\gamma}(\beta)$ and $CB(\alpha)=CB(\beta)$, we have that \[c_{\alpha}(\omega^{m_1}+\dots+\omega^{m_k},{\ell})=c_{\beta}(\omega^{m_1}+\dots+\omega^{m_k},{\ell})\]
for all $1 \leq k \leq n$ with $k \neq CNF_{\gamma}(\alpha)$ and for all $0 \leq {\ell}\leq m_k$, where $c_{\alpha}(\cdot,\cdot)$ is as in Item b.
\end{itemize}
Consequently, for a canonical coloring $\cc:[\gamma]^2 \rightarrow \{0,1\}$, there exists a function $\tcc(i,j;k,\ell)$ defined for $1 \leq k \neq i \leq n$ and $0 \leq j \leq m_i$ and $0 \leq \ell \leq m_k$ such that
\[ \tcc(i,j;k,{\ell})=c_{\alpha}(\omega^{m_1}+\dots+\omega^{m_k},{\ell})\]
where $\alpha \in \gamma$ is any ordinal with $CNF_{\gamma}(\alpha)=i$, $CB_{\gamma}(\alpha)=j$. Observe that the following are equivalent for a canonical coloring $\cc:[\gamma]^2 \rightarrow \{0,1\}$.
\begin{itemize}
\item $\tcc(i,j;k,{\ell})=c$
\item For all $\alpha$ with $CNF_{\gamma}(\alpha)=i$ and $CB(\alpha)=j$, there exists $r \in \mathbb{N}$ such that $\cc(\{\alpha,\beta\})=c$ for every $\beta \in F\left(\omega^{m_1}+\dots+\omega^{m_k}\right)^r_{{\ell}}$.
\end{itemize}
In proofs, we shall use this equivalence whenever we need to use that $\tcc(i,j;k,{\ell})=c$. As was the case with $\omega$-homogeneous colorings, there always exist skeletons for which the induced colorings are canonical.
\begin{fact}\label{mermelcanonical}\cite[Proposition 3.11]{Mermelstein19} For every coloring $\mathfrak{c}: [\gamma]^2 \rightarrow \{0,1\}$ there exists a skeleton $I \subseteq \gamma$ such that $\mathfrak{c}_{I}: [\gamma]^2 \rightarrow \{0,1\}$ is canonical.
\end{fact}
\noindent \textbf{Important remark.} In \cite{Mermelstein19}, the original definition of a canonical coloring only requires $\ell < CB(\theta)$ in Item b and $\ell < m_k$ in Item c. However, analyzing the proof of Fact \ref{mermelcanonical}, one sees that the proof still goes through for this modified definition. (Though, one needs to be careful while using this definition since $\ell$ cannot be $m_k$ for $k=n$ in the case that $\gamma$ is not successor and $\gamma=\omega^{m_1}+\omega^{m_2}+\dots+\omega^{m_n}$.)\\

We remark that it follows from Lemma \ref{omegahomcolor} and Fact \ref{mermelcanonical} that, in order to prove an inequality of the form $R(\alpha,\beta) \leq \gamma$, it suffices to prove that any $\omega$-homogeneous canonical coloring of $\gamma$ has a red homogeneous copy of $\alpha$ or a blue homogeneous copy of $\beta$. The reason is that, given a coloring $\cc:[\gamma]^2 \rightarrow \{0,1\}$, we can first find a skeleton $I \subseteq \gamma$ for which $\cc_I:[\gamma]^2 \rightarrow \{0,1\}$ is $\omega$-homogeneous and then, find a skeleton $J \subseteq \gamma$ for which $(\cc_I)_J: [\gamma]^2 \rightarrow \{0,1\}$ is both canonical and $\omega$-homogeneous. (For the latter claim, observe that the induced coloring of an $\omega$-homogeneous coloring with respect to a skeleton is $\omega$-homogeneous.) But then, any homogeneous subset of $\gamma$ with respect to $(\cc_I)_J$ can be pulled back to a homogeneous subset of $\gamma$ with respect to $\cc$ of the same order type.

Before we conclude this section, let us introduce some notation and state a lemma for later use. For an ordinal $\theta < \omega^{\omega}$, we define $\node{i}{j}{\theta}$ to be the set
\[ \node{i}{j}{\theta}=\{\alpha \in \theta: CNF_{\theta}(\alpha)=i,\  CB(\alpha)=j\} \]

\begin{lemma}\label{omegasquaredlevels} Let $n \geq 2$ be an integer and let $\mathfrak{c}: [\omega^2]^2 \rightarrow \{0,1\}$ be a normal $\omega$-homogeneous coloring with no red homogeneous closed copy of $\omega+n$ and no blue homogeneous closed copy of $\omega$. Then 
\begin{itemize}
\item[(a)] $\hcc(1,1,0)=1$ or
\item[(b)] For every $i < \omega$, we have that
\[ \{\beta \in \node{1}{1}{\omega^2} :\ \{\alpha \in W_i:\ \mathfrak{c}(\{\alpha,\beta\})=1\} \text{ is cofinal in } W_i\}\]
is cofinal in $\node{1}{1}{\omega^2}$, where $W_i=\{\omega \cdot i + m : 0 < m < \omega\}$.
\end{itemize}
\end{lemma}
\begin{proof} Assume towards a contradiction that $\hcc(1,1,0)=0$ and that, for some $i<\omega$,
\[ \{\beta \in \node{1}{1}{\omega^2} :\ \{\alpha \in W_i:\ \mathfrak{c}(\{\alpha,\beta\})=1\} \text{ is cofinal in } W_i\}\]
is not cofinal in $\node{1}{1}{\omega^2} \cong \omega$ and hence, is finite. Since $W_i \cong \omega$, the complement of this set
\[ \{\beta \in \node{1}{1}{\omega^2} :\ \{\alpha \in W_i:\ \mathfrak{c}(\{\alpha,\beta\})=1\} \text{ is finite}\}\]
is cofinite in $\node{1}{1}{\omega^2}$. Thus there exist ordinals $\omega \cdot (i+1) =\beta_0 < \beta_1 < \dots < \beta_{n-1}$ in $\node{1}{1}{\omega^2}$ such that $\{\alpha \in W_i:\ \mathfrak{c}(\{\alpha,\beta_j\})=1\}$ is finite for each $1 \leq j \leq n-1$. Hence the set
\[ H_i=\{\alpha \in W_i:\ \mathfrak{c}(\{\alpha,\beta_j\})=0\ \text{ for all } 1 \leq j \leq n-1\}\]
is cofinite in $W_i$. Since $\mathfrak{c}$ is $\omega$-homogeneous and there exists no blue homogeneous copy of $\omega$, the sets $W_i$ and $\{\beta_0,\beta_1,\dots,\beta_{n-1}\}$ are both red homogeneous. Also, $\hcc(1,1,0)=0$ implies that $c(\{\alpha,\omega \cdot (i+1)\})=0$ for all $\alpha \in W_i$. It follows that the set
\[ H_i \cup \{\omega \cdot (i+1), \beta_1, \beta_2, \dots, \beta_{n-1}\} \]
is a red homogeneous closed copy of $\omega+n$, which is a contradiction.
\end{proof}

\section{A lower bound}

Before we proceed to prove Theorem \ref{mainresult-lower}, we shall construct a special triangle-free graph whose vertices are subsets of ordinals and whose edges shall induce a coloring that witnesses $\gamma=\omega^2 \cdot n + \omega \cdot (R(n,3)-n)+(n-1)  \nrightarrow (\omega+n,3)^2$.

Let $n \geq 3$ be an integer and set $K=R(n,3)-n$. For each integer $1 \leq i \leq n$, set
\begin{align*}
A_1&=\node{1}{0}{\gamma}= \left\{\alpha + 1\ :\ \alpha < \omega^2\right\} \cup \{0\}\\ A_i&=\node{i}{0}{\gamma}=\left\{\omega^2 \cdot (i-1) + \alpha + 1\ :\ \alpha < \omega^2\right\} & \text{ if } i \neq 1\\
B_i&=\node{i}{1}{\gamma}=\left\{\omega^2 \cdot (i-1) + \omega \cdot (k+1)\ :\ k < \omega\right\}\\
L_i&=\node{i}{2}{\gamma}=\left\{\omega^2 \cdot i\right\}
\end{align*}
and for each integer $n < i \leq n+K$, set
\begin{align*}
C_i&=\node{i}{0}{\gamma}=\left\{\omega^2 \cdot n + \omega \cdot (i-n-1) + \alpha + 1\ :\ \alpha < \omega\right\}\\
L_i&=\node{i}{1}{\gamma}=\left\{\omega^2 \cdot n + \omega \cdot (i-n)\right\}
\end{align*}
In addition, set
$$ R=\{\omega^2 \cdot n + \omega \cdot K + m: 1 \leq m \leq n-2\}$$
It is easily seen that $$V=\{A_i,B_i,L_i: 1 \leq i \leq n\} \cup \{C_i,L_i: n < i \leq n+K\} \cup R$$ is a partition of the ordinal $\gamma=\omega^2 \cdot n + \omega \cdot (R(n,3)-n) + (n-1)$. Observe that, by the definition of Ramsey numbers, there exists a coloring $$\mathfrak{r}: \left[\{L_i: 1 \leq i < R(n,3)\}\right]^2 \rightarrow \{0,1\}$$ with no red homogeneous sets of size $n$ and no blue homogeneous set of size $3$. It is well-known that $R(n-1,3)<R(n,3)-1$ and so, there exists a red homogeneous set of size $n-1$ with respect to $\mathfrak{r}$. By relabeling if necessary, we may assume without loss of generality that the set $\{L_1,L_2,\dots,L_{n-1}\}$ is red homogeneous. Fix such a coloring $\mathfrak{r}$.

We shall next construct a graph $\GG_n$ on the vertex set $V$ using $\mathfrak{r}$. Define the edge sets $E_1$, $E_2$, $E_3$ and $E_4$ as follows.
\begin{align*}
E_1&=\{\{L_i,A_j\},\{A_i,B_j\}: 1 \leq i < j \leq n\}\cup\{\{A_i,B_i\},\{B_i,L_i\}: 1 \leq i \leq n\}\\
E_2&=\{\{C_i,L_i\}: n < i < n+K\}\\
E_3&=\left\{\{L_i,L_j\}: \mathfrak{r}\left(\{L_i,L_j\}\right)=1,\ 1 \leq i \neq j < R(n,3)\right\}\\
E_4&=\{\{X,A_i\}: X \in W,\ 1 \leq i \leq n\}\}
\end{align*}
where $W=\{C_{n+1},C_{n+2},\dots,C_{n+K},L_{n+K},R\}$. Consider the graph $$\GG_n=(V, E_1 \cup E_2 \cup E_3 \cup E_4)$$ We shall not attempt to perform the impossible task of drawing a diagram representation of $\GG_n$, simply because there is no known way to find such a map $\mathfrak{r}$ for an arbitrary $n$. However, in Figure \ref{figuremain}, we do provide a diagram representation of the subgraph $(V,E_1 \cup E_2)$ for arbitrary $n$ so that the reader may follow the arguments on this diagram if necessary. In Figure \ref{figurefor3}, a diagram representation of $\GG_3$ is given for some appropriate choice of $\mathfrak{r}$.

\newpage

\begin{figure}[h!]
  \includegraphics[scale=0.65]{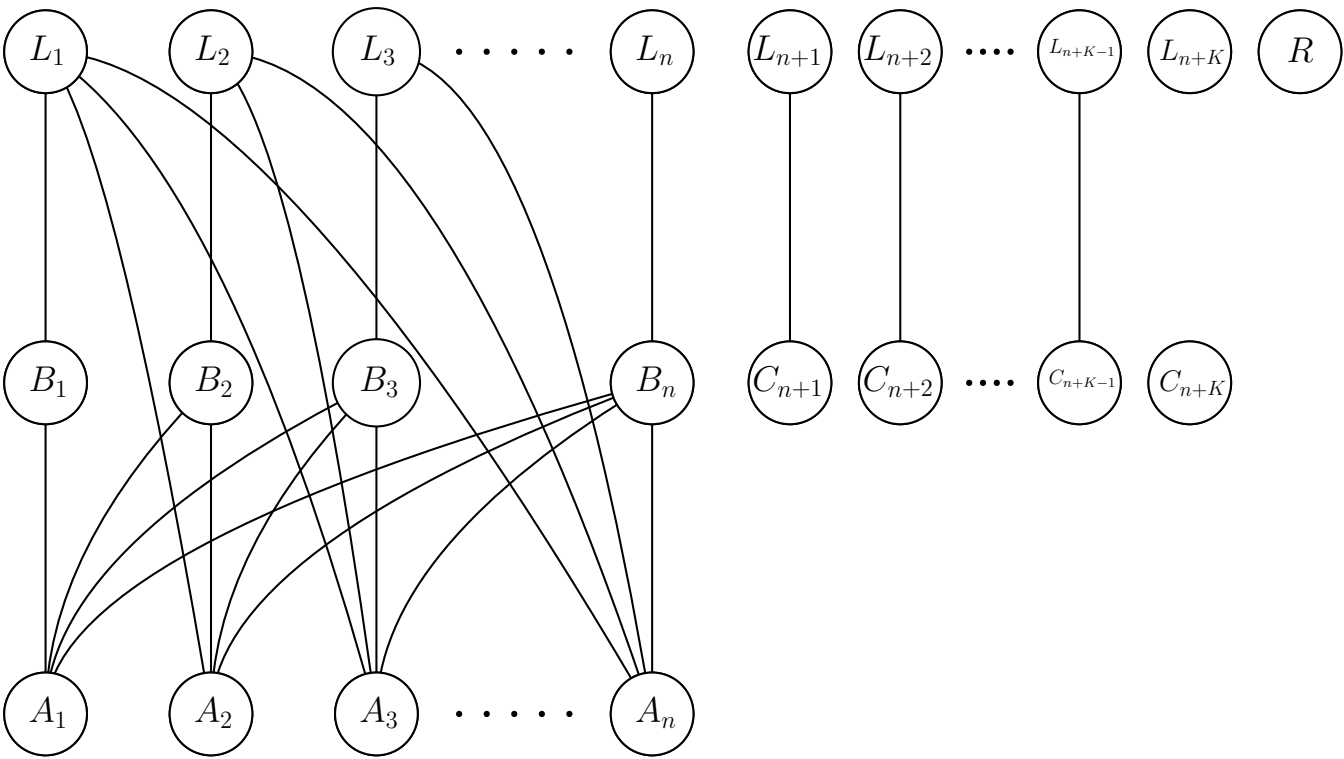}
  \caption{A diagram representation of $(V,E_1 \cup E_2)$.}
  \label{figuremain}
\end{figure}

\newpage

\begin{figure}[h!]
  \includegraphics[scale=0.70]{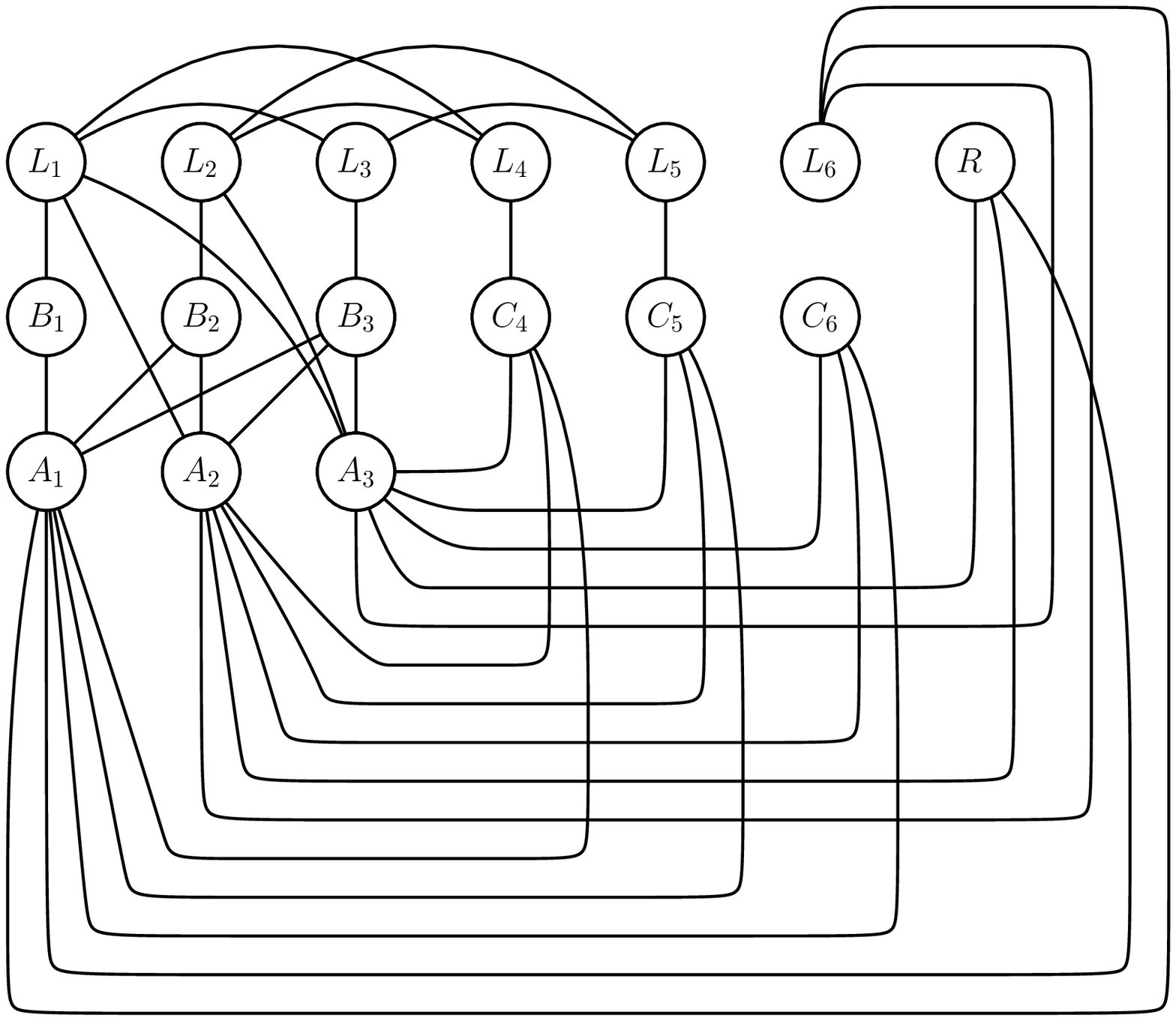}
  \caption{A diagram representation of $\mathbf{G}_3$ for a choice of $\mathfrak{r}$.}
  \label{figurefor3}
\end{figure}

\begin{lemma}\label{trianglefreelemma} The graph $\GG_n$ is triangle-free.
\end{lemma}

\begin{proof} We shall first prove that $(V,E_1)$ is triangle-free. Assume towards a contradiction that there exists a triangle $T \subseteq V$ in the graph $(V,E_1)$. Since no two $B_i$'s are adjacent and no two $L_i$'s are adjacent in $(V,E_1)$, we must have that there exists $1 \leq j \leq n$ with $A_j \in T$. On the other hand, the set of neighbors of $A_j$ in $(V,E_1)$ is
\[ \{B_k: j \leq k \leq n\} \cup \{L_i: 1 \leq i < j\}\]
Moreover, for $1 \leq i,k \leq n$, we have that $L_i$ is adjacent to $B_k$ if and only if $i=k$. It follows that no two neighbors of $A_j$ are adjacent, which is a contradiction.

Having proven that $(V,E_1)$ is triangle-free, it is easily verified that $(V,E_1 \cup E_2)$ is triangle-free. This follows from the fact that the edges in $E_2$ are
\begin{itemize}
\item not incident with vertices that are incident to the edges in $E_1$, and
\item not incident with each other.
\end{itemize}

We shall next prove that $(V,E_1 \cup E_2 \cup E_3)$ is triangle-free. Suppose that there exists a triangle $T=\{X,Y,Z\} \subseteq V$ in the graph $(V,E_1 \cup E_2 \cup E_3)$. Since $(V,E_1 \cup E_2)$ is triangle-free, we must have that some edge in $E_3$ are incident to vertices in $T$, say, $X=L_i$ and $Y=L_j$ for some $1 \leq i < j < R(n,3)$. Since $\{L_1,\dots,L_{n-1}\}$ was arranged to be red homogeneous with respect to $\mathfrak{r}$, we must have $n \leq j$.

Then, by construction, we have that $Z=B_n$, $Z=C_j$ or $Z=L_k$ for some $1 \leq k \neq j < R(n,3)$. The first case leads to a contradiction as the only neighbor of $B_n$ among $L_m$'s in this graph is $L_n$. The second case leads to a contradiction as $C_j$'s only neighbor in this graph is $L_j$. The third case leads to a contradiction because there are no edges between $L_m$'s in the graph $(V,E_1 \cup E_2)$ and consequently, the third case happening would imply that all the edges of this triangle are from $E_3$, in which case we would have $\mathfrak{r}(\{L_i,L_j\})=\mathfrak{r}(\{L_j,L_k\})=\mathfrak{r}(\{L_i,L_k\})=1$, creating a blue homogeneous set of size $3$ with respect to $\mathfrak{r}$. Thus $(V,E_1 \cup E_2 \cup E_3)$ is triangle-free.

Finally, we shall prove that $\GG_n=(V,E_1 \cup E_2 \cup E_3 \cup E_4)$ is triangle-free. Suppose that there exists a triangle $T =\{X,Y,Z\} \subseteq V$ in $\GG_n$. Since $(V,E_1 \cup E_2 \cup E_3)$ is triangle-free, we must have that some edge in $E_4$ are incident to vertices in $T$, say, $X \in W$ and $Y=A_i$ for some $1 \leq i \leq n$. The set of neighbors of $X$ is a subset of $\{L_j: n<j<R(n,3)\}$ and the set of neighbors of $Y$ is a subset of $\{L_j: 1 \leq j <i \leq n\} \cup \{B_k: i \leq k \leq n\} \cup W$. However, these sets do not intersect and hence, we have a contradiction. Therefore, $\GG_n$ is triangle-free.
\end{proof}

We are now ready to prove the first main result.

\begin{proof}[Proof of Theorem \ref{mainresult-lower}] Let $n \geq 3$ be a positive integer and set $$\gamma=\omega^2 \cdot n + \omega \cdot (R(n,3)-n)+(n-1)$$ Consider the coloring $c: [\gamma]^2 \rightarrow \{0,1\}$ given by $c(\{\alpha,\beta\})=1$ if and only if the vertices containing $\alpha$ and $\beta$ in $\mathbf{G}_n$ are adjacent.

Since $\mathbf{G}_n$ is triangle-free by Lemma \ref{trianglefreelemma}, there does not exist a blue homogeneous copy of $3=\{0,1,2\}$. We shall next show that there exists no red homogeneous closed copy $X$ of $\omega+n$. Assume to the contrary that there exists such a set $X \subseteq \gamma$. Observe that, by definition, the vertices to which the elements of $X$ belong are not adjacent in $\mathbf{G}_n$. Let $h: \omega+n \rightarrow X$ be the order-homeomorphism and let us denote $h(\alpha)$ by $\dot{\boldsymbol{\alpha}}$ for all $\alpha \in \omega+n$. So we can write $X$ as
\[ \dot{\mathbf{0}}<\dot{\mathbf{1}}<\dots<\dot{\mathbf{\boldsymbol{\omega}}}<\dot{\mathbf{\boldsymbol{\omega}+1}}<\dots<\dot{\mathbf{\boldsymbol{\omega}+n-1}}\]
Since $\dot{\boldsymbol{\omega}}$ is a limit ordinal, we have that $\dot{\boldsymbol{\omega}} \in B_i$ for some $1 \leq i \leq n$, or, $\dot{\boldsymbol{\omega}} \in L_i$ for some $1 \leq i \leq R(n,3)$. We now analyze these cases.

Suppose that $\dot{\boldsymbol{\omega}} \in B_i$ for some $1 \leq i \leq n$. Then, for cofinitely many $k \in \omega$, we have $ \dot{\mathbf{k}} \in A_i$. On the other hand, $A_i$ and $B_i$ are adjacent in $\GG_n$, which leads to a contradiction.

Suppose that $\dot{\boldsymbol{\omega}} \in L_i$ for some $1 \leq i \leq n$. Then, for cofinitely many $k \in \omega$, we have $ \dot{\mathbf{k}} \in A_i$ or $ \dot{\mathbf{k}} \in B_i$. But, as $B_i$ and $L_i$ are adjacent, we obtain that $ \dot{\mathbf{k}} \in A_i$ for cofinitely many $k \in \omega$. Recall that
\begin{itemize}
\item $A_i$ is adjacent to each vertex in $\{B_j: i \leq j \leq n\} \cup W$, and
\item $L_i$ is adjacent to each vertex in $\{A_j: i<j \leq n\}$.
\end{itemize}
So the vertices in $\{A_j: i<j \leq n\} \cup \{B_j: i \leq j \leq n\} \cup W$ cannot contain the elements of $X$ greater than $\dot{\mathbf{\boldsymbol{\omega}}}$. It follows that
\[\{\dot{\mathbf{\boldsymbol{\omega}}},\dot{\mathbf{\boldsymbol{\omega}+1}},\dot{\mathbf{\boldsymbol{\omega}+2}},\dots,\dot{\mathbf{\boldsymbol{\omega}+n-1}}\} \subseteq \bigcup_{j=i}^{R(n,3)-1} L_j\]
Since each $L_j$ is a singleton and the set on the left-hand side is red homogeneous, we obtain that $\{L_j: i \leq j < R(n,3)\}$ has a subset of size $n$, no two vertices of which are adjacent. Recall that the edges between $L_j$'s in $\GG_n$ come from $E_3$. Consequently, there exists a red homogeneous set of size $n$ with respect to the coloring $\mathfrak{r}$, which is a contradiction.

Suppose that  $\dot{\boldsymbol{\omega}} \in L_i$ for some $n < i < R(n,3)$. Then, for cofinitely many $k \in \omega$, we have $ \dot{\mathbf{k}} \in C_i$. On the other hand, $C_i$ and $L_i$ are adjacent in $\GG_n$, which leads to a contradiction.

Finally, suppose that $\dot{\boldsymbol{\omega}} \in L_{R(n,3)}$. In this case, we must have
\[\{\dot{\mathbf{\boldsymbol{\omega}+1}},\dot{\mathbf{\boldsymbol{\omega}+2}},\dots,\dot{\mathbf{\boldsymbol{\omega}+n-1}}\} \subseteq R\]
This is a contradiction as the left-hand side has $n-1$ elements, whereas, the right-hand side has $n-2$ elements. We obtained contradictions in all cases. Therefore, there exists no such set $X$ and so $$\omega^2 \cdot n + \omega \cdot (R(n,3)-n)+(n-1)  \nrightarrow (\omega+n,3)^2$$
This completes the proof.
\end{proof}

\section{An upper bound}

In this section, we shall prove Theorem \ref{mainresult-upper}. In order to do this, we will need several technical lemmas. For the following lemmas, fix integers $n,K \geq 2$, the ordinal $$\gamma=\omega^2 \cdot n + \omega \cdot K +1$$ and a canonical $\omega$-homogeneous coloring $\mathfrak{c}: [\gamma]^2 \rightarrow \{0,1\}$ with no red homogeneous closed copy of $\omega+n$ and no blue homogeneous copy of $3=\{0,1,2\}$. Recall that, since $\cc$ is canonical, there exist functions $\hcc$ and $\tcc$ as in Section 2.

The proof of Theorem \ref{mainresult-upper} will be a convoluted case-by-case proof that uses these lemmas which essentially show that certain values of $\tcc$ and $\hcc$ are automatically determined by the non-existence of a red homogeneous $\omega+n$ and a blue homogeneous $3$. We shall see that, intuitively speaking, some of the patterns in the graph $\GG_n$ were unavoidable and had to appear if we are to avoid certain homogeneous sets. To keep track of what is going on, the reader may want to ``visualize" the statements and arguments of these lemmas. In order to do this, the reader may pretend that we are constructing a directed graph on the partition $$\gamma = \bigsqcup_{1 \leq i \leq n,\ 0 \leq j \leq 2} \node{i}{j}{\gamma}\ \ \ \sqcup \bigsqcup_{n+1 \leq i \leq n+K,\ 0 \leq j \leq 1} \node{i}{j}{\gamma}$$ of $\gamma$ by putting an edge
\begin{itemize}
\item from $\node{i}{j}{\gamma}$ to $\node{k}{\ell}{\gamma}$ if $\tcc(i,j;k,\ell)=1$.
\item from $\node{i}{2}{\gamma}$ to $\node{i}{\ell}{\gamma}$ if $\hcc(i,2,\ell)=1$.
\item from $\node{i}{1}{\gamma}$ to $\node{i}{0}{\gamma}$.
\end{itemize}
We would like to note that $\tcc(i,j;k,\ell)=1$ does not imply that every pair of ordinals coming from the corresponding vertices have the color $1$ under $\cc$. It only implies that ``most" pairs have the color $1$. We would also like to remark that the edges from $\node{i}{1}{\gamma}$ to $\node{i}{0}{\gamma}$ are automatically added regardless of the value of $\hcc$, due to Lemma \ref{omegasquaredlevels}, which says that either $\hcc(i,1,0)=1$ or there are ``many" pairs coming from the corresponding vertices for which $\cc$ has value $1$. Finally, we wish to emphasize that whether we are using a directed graph or an undirected graph has absolutely no role in the proofs. Indeed, our arguments do not refer to any graphs at all. We are simply suggesting this ``supplementary" approach if the reader wishes to do more than line-by-line proof checking. With this graph interpretation in mind, the following lemma prevents the existence of certain triangles in this graph. The first four items of this lemma essentially appeared in \cite[Lemma 4.3]{Mermelstein19}. Nevertheless, we include the proofs for self-containment.
\begin{lemma}\label{notrianglelemma} Let $1 \leq k \leq n$ and $0 \leq \ell \leq 2$.
\begin{itemize}
\item[(a)] For every $1 \leq i \neq k \leq n$ and $0 \leq j \leq 2$, $$\tcc(i,j;k,0)=0 \text{ or } \tcc(i,j;k,1)=0.$$
\item[(a$'$)] For every $n+1 \leq i \leq n+K$ and $0 \leq j \leq 1$, $$\tcc(i,j;k,0)=0 \text{ or } \tcc(i,j;k,1)=0.$$
\item[(b)] For every $1 \leq i \neq k  \leq n$, $$\tcc(i,0;k,\ell)=0 \text{ or } \tcc(i,1;k,\ell)=0.$$
\item[(c)] For every $1 \leq i \leq n$, $\hcc(i,2,0)=0$ or $\hcc(i,2,1)=0$.
\item[(d)] For every $n+1 \leq i \neq m \leq n+K$, if $\tcc(m,0;k,\ell)=\tcc(i,0;k,\ell)=1$, then $\tcc(m,0;i,0)=0$.
\end{itemize}
\end{lemma}
\begin{proof} To prove (a) and (a$'$), assume to the contrary that $\tcc(i,j;k,0)=1$ and $\tcc(i,j;k,1)=1$ for some such $i$ and $j$. By definition of $\tcc$, we know that for all $\theta \in \node{i}{j}{\gamma}$ there exist $r_1, r_2 \in \mathbb{N}^+$ such that for all $\beta_1 \in F(\omega^2 \cdot k)^{r_1}_{0}$ and $\beta_2 \in F(\omega^2 \cdot k)^{r_2}_{1}$ we have $\cc(\{\theta,\beta_1\})=\cc(\{\theta,\beta_2\})=1$. Let $r=\max\{r_1,r_2\}$ and fix $\theta \in \node{i}{j}{\gamma}$. Applying Lemma \ref{omegasquaredlevels} to the closed copy $\{\beta \in \gamma: \beta <^* \omega^2 \cdot k\}$ of $\omega^2$, we may split into two cases.

Suppose that $\hcc(k,1,0)=1$. Choose some $\beta \in F(\omega^2 \cdot k)^{r}_{1}$ and $\beta' \in F(\omega^2 \cdot k)^{r}_{0}$ with $\beta' <^* \beta$. Then $\cc(\{\beta,\beta'\})=1$ and hence, the set $\{\theta,\beta,\beta'\}$ is a blue homogeneous copy of $3$, which is a contradiction.

Suppose that, for every $i < \omega$,
\[ \{\beta \triangleleft^* \omega^2 \cdot k :\ \{\alpha \in W_i:\ \mathfrak{c}(\{\alpha,\beta\})=1\} \text{ is cofinal in } W_i\}\]
is cofinal in $\{\beta \in \gamma: \beta \triangleleft^* \omega^2 \cdot k\}$, where $W_i=\{\omega^2 \cdot (k-1) + \omega \cdot i + m : 0 < m < \omega\}$. In particular, this claim holds for $i=r$. We can then find some $\beta \in F(\omega^2 \cdot k)^{r}_{1}$ and $\beta' \in F(\omega^2 \cdot k)^{r}_{0}$ with $\cc(\{\beta,\beta'\})=1$. In this case, the set $\{\theta,\beta,\beta'\}$ is a blue homogeneous copy of $3$, which is a contradiction. This completes the proof of (a) and (a$'$).

To prove (b), assume to the contrary that $\tcc(i,0;k,\ell)=1 \text{ and } \tcc(i,1;k,\ell)=1$ for some $1 \leq i \neq k  \leq n$ and $0 \leq \ell \leq 2$. Applying Lemma \ref{omegasquaredlevels} to the closed copy $\{\beta \in \gamma: \beta <^* \omega^2 \cdot i\}$ of $\omega^2$, we see that, in either case, there exist some $\beta \in \node{i}{1}{\gamma}$ and $\beta' \in \node{i}{0}{\gamma}$ with $\cc(\{\beta,\beta'\})=1$, .

Then, by the definition of $\tcc$, there exists $r_1 \in \mathbb{N}^+$ such that $\cc(\{\beta,\theta\})=1$ whenever $\theta \in F(\omega^2 \cdot k)^{r_1}_{\ell}$; and there exists $r_2 \in \mathbb{N}^+$ such that $\cc(\{\beta',\theta\})=1$ whenever $\theta \in F(\omega^2 \cdot k)^{r_2}_{\ell}$. Choose $\theta \in F(\omega^2 \cdot k)^{r}_{\ell}$ where $r=\max\{r_1,r_2\}$. Then the set $\{\theta,\beta,\beta'\}$ is a blue homogeneous copy of $3$, which is a contradiction. This completes the proof of (b).

To prove (c), assume to the contrary that $\hcc(i,2,0)=1$ and $\hcc(i,2,1)=1$ for some $1 \leq i \leq n$. As before, applying Lemma \ref{omegasquaredlevels} to the closed copy $\{\beta \in \gamma: \beta <^* \omega^2 \cdot i\}$ of $\omega^2$ gives us $\beta \in \node{i}{1}{\gamma}$ and $\beta' \in \node{i}{0}{\gamma}$ such that $\cc(\{\beta,\beta'\})=1$. Then the set $\{\omega^2 \cdot i,\beta,\beta'\}$ is a blue homogeneous copy of $3$, which is a contradiction. This completes the proof of (c)

To prove (d), assume to the contrary that $$\tcc(m,0;i,0)=\tcc(m,0;k,\ell)=\tcc(i,0;k,\ell)=1$$ for some such $i$ and $m$. Let $\alpha \in \node{m}{0}{\gamma}$. Since $\tcc(m,0;i,0)=1$, as before, we can find $\alpha \in \node{m}{0}{\gamma}$ and $\beta \in \node{i}{0}{\gamma}$ such that $\cc(\{\alpha,\beta\})=1$. The other equations now imply that there exist $r_1,r_2 \in \mathbb{N}^+$ such that for all $\beta_1 \in F(\omega^2 \cdot k)^{r_1}_{\ell}$ and $\beta_2 \in F(\omega^2 \cdot k)^{r_2}_{\ell}$ we have $\cc(\{\alpha,\beta_1\})=\cc(\{\beta,\beta_2\})=1$. Choose $\theta \in F(\omega^2 \cdot k)^{r}_{\ell}$ where $r=\max\{r_1,r_2\}$. Then the set $\{\alpha,\beta,\theta\}$ is a blue homogeneous copy of $3$, which is a contradiction. This finishes the proof.\end{proof}

\begin{lemma}\label{omegaplusnlemma}Let $1 \leq k \leq n$ and $0 \leq \ell \leq 1$. \begin{itemize}
\item[(a)]  For all $k < i \leq n$ and $0 \leq j \leq 1$,\\ if $\hcc(k,2,\ell)=0$, then $\tcc(i,j;k,\ell)=1$ or $\tcc(i,j;k,2)=1$.
\item[(b)]  For all $n+1 \leq i \leq n+K$ and $j=0$,\\ if $\hcc(k,2,\ell)=0$, then $\tcc(i,j;k,\ell)=1$ or $\tcc(i,j;k,2)=1$.
\end{itemize}
\end{lemma}
\begin{proof} We shall prove both parts at once. Assume to the contrary that $$\hcc(k,2,\ell)=\tcc(i,j;k,\ell)=\tcc(i,j;k,2)=0$$ for some such $i$ and $j$. Then, by the definition of $\tcc$, for every $\alpha \in \node{i}{j}{\gamma}$ there exists $r_{\alpha} \in \mathbb{N}$ such that $$\cc(\{\alpha,\beta\})=\cc(\{\alpha,\omega^2 \cdot k\})=0$$ for all $\beta \in F(\omega^2 \cdot k)^{r_{\alpha}}_{\ell}$. Since $\cc$ is $\omega$-homogeneous and there is no homogeneous copy of $3$, we can find a red homogeneous set $\{\alpha_1, \alpha_2, \dots, \alpha_{n-1}\} \subseteq \node{i}{j}{\gamma}$ with $\alpha_1 < \dots < \alpha_{n-1}$. Set
\[ r = \max\{r_{\alpha_1},r_{\alpha_2},\dots,r_{\alpha_{n-1}}\}\]
It is now straightforward to check that the set
\[ F(\omega^2 \cdot k)^{r}_{\ell} \cup \{\omega^2 \cdot k, \alpha_1,\alpha_2,\dots,\alpha_{n-1}\}\]
is a red homogeneous closed copy of $\omega+n$, which leads to a contradiction.
\end{proof}

\begin{lemma}\label{alconnections} Let $1 \leq k < n$.
\begin{itemize}
\item[(a)] For any $k< i \leq n$ and $0 \leq j \leq 1$, we have that $$\tcc(i,j;k,0)=1 \text{ or } \tcc(i,j;k,1)=1 \text{ or }\tcc(i,j;k,2)=1.$$
\item[(b)] For any $n+1 \leq i \leq n+K$ and $j=0$, we have that $$\tcc(i,j;k,0)=1 \text{ or } \tcc(i,j;k,1)=1 \text{ or }\tcc(i,j;k,2)=1.$$
\end{itemize}
\end{lemma}
\begin{proof} As before, we shall prove both parts at once. Assume towards a contradiction that, for some, for some such $i$ and $j$, $$\tcc(i,j;k,0)=\tcc(i,j;k,1)=\tcc(i,j;k,2)=0$$ By Lemma \ref{notrianglelemma}.c, we have $\hcc(k,2,0)=0$ or $\hcc(k,2,1)=0$. In both cases, we get a contradiction by Lemma \ref{omegaplusnlemma}.
\end{proof}

\begin{lemma}\label{lbconnection} For any $1 \leq i < n$, either $\hcc(i,2,0)=1$ or $\hcc(i,2,1)=1$.
\end{lemma}
\begin{proof} It is clear by Lemma \ref{notrianglelemma}.c that we do not have $\hcc(i,2,0)=1$ and $\hcc(i,2,1)=1$. So assume towards a contradiction that $\hcc(i,2,0)=0$ and $\hcc(i,2,1)=0$ for some $1 \leq i \leq n$.

It follows from Lemma \ref{notrianglelemma}.b, $\tcc(i+1,j;i,2)=0$ for some $0 \leq j \leq 1$. Now, by Lemma \ref{notrianglelemma}.a, $\tcc(i+1,j;i,\ell)=0$ for some $0 \leq \ell \leq 1$. But then, we have that $\hcc(i,2,\ell)=\tcc(i+1,j;i,\ell)=\tcc(i+1,j;i,2)=0$, which contradicts Lemma \ref{omegaplusnlemma}.a.
\end{proof}

Before we state the next lemma, we will introduce some notation in light of Lemma \ref{lbconnection}. For each $1 \leq i < n$, let us denote
\begin{itemize}
\item the (unique) pair $(i,j)$ for which $\hcc(i,2,j)=0$ by $A_i$ and
\item the (unique) pair $(i,j)$ for which $\hcc(i,2,j)=1$ by $B_i$
\end{itemize}
We shall also denote
\begin{itemize}
\item the pair $(i,2)$ by $L_i$ for each $1 \leq i \leq n$, and
\item the pair $(i,1)$ by $L_i$ for each $n+1 \leq i \leq n+K$.
\end{itemize}
Using the ideas in the proof of Lemma \ref{omegaplusnlemma} and this newly introduced notation, we next prove a technical lemma that will be used multiple times in the main proof.

\begin{lemma}\label{omegaplusnlemma2} Let $1 \leq k \leq n$ and $0 \leq \ell \leq 1$ and let $k < i_1 < i_2 < \dots < i_{n-1} \leq n+K$. If we have $$\tcc(L_{i_{t'}};L_{i_{t}})=\tcc(L_{i_t};k,\ell)=\tcc(L_{i_t};L_k)=0$$ for all $1 \leq t < t' \leq n-1$, then $\hcc(k,2,\ell)=1$.
\end{lemma}
\begin{proof} Assume to the contrary that
\begin{itemize}
\item[(i)] $\tcc(L_{i_{t'}};L_{i_{t}})=0$ for all $1 \leq t < t' \leq n-1$,
\item[(ii)] $\tcc(L_{i_t};k,\ell)=0$ for all $1 \leq t \leq n-1$,
\item[(iii)] $\tcc(L_{i_t};L_k)=0$ for all $1 \leq t \leq n-1$, and
\item[(iv)] $\hcc(k,2,\ell)=0$.
\end{itemize}
Let $\alpha_t$ be the unique element of $\mathbf{[\boldsymbol{\gamma};L_{i_t}]}$ for each $1 \leq t \leq n-1$. Using (i) and (iii), one can easily show that $$\{\omega^2 \cdot k,\alpha_1,\alpha_2,\dots,\alpha_{n-1}\}$$ is red homogeneous. It follows from (ii) and the definition of $\tcc$ that, for every $1 \leq t \leq n-1$, there exists $r_{t} \in \mathbb{N}$ such that $\cc(\{\alpha_t,\beta\})=0$ for all $\beta \in F(\omega^2 \cdot k)^{r_t}_{\ell}$. Set $r = \max\{r_{1},r_{2},\dots,r_{n-1}\}$. It is now straightforward to check using the previous observations and (iv) that the set
\[ F(\omega^2 \cdot k)^{r}_{\ell} \cup \{\omega^2 \cdot k,\alpha_1,\alpha_2,\dots,\alpha_{n-1}\}\]
is a red homogeneous closed copy of $\omega+n$, which leads to a contradiction.
\end{proof}

\begin{lemma}\label{notobconnection}Let $1 \leq k < n$. Then
\begin{itemize}
\item[a.] For any $k < i \leq n$, we have $\tcc(i,0;B_k)=\tcc(i,1;B_k)=0$.
\item[b.] For any $n+1 \leq i \leq n+K$, we have $\tcc(i,0;B_k)=0$.
\end{itemize}
\end{lemma}
\begin{proof} We will prove both parts at once. Assume towards a contradiction that $\tcc(i,j;B_k)=1$ for some such $i$ and $j$. Then it follows from Lemma \ref{notrianglelemma}.a that $\tcc(i,j;A_k)=0$.

By the assumption that $\tcc(i,j;B_k)=1$, for every $\alpha \in \node{i}{j}{\gamma}$, there exists $r \in \mathbb{N}^+$ such that $\cc(\{\alpha,\beta\})=1$ for every $\beta \in F(\omega^2 \cdot k)^r_{\ell}$ where $B_k=(k,\ell)$. Choose such $\alpha$ and $\beta$. Since $\hcc(k,2,\ell)=1$, if it were the case that $\tcc(i,j;L_k)=1$, then the set $\{\alpha,\beta,\omega^2 \cdot k\}$ would be a blue homogeneous copy of $3$. Thus $\tcc(i,j;L_k)=0$.  But this contradicts Lemma \ref{omegaplusnlemma}, as we have $\tcc(i,j;L_k)=\tcc(i,j;A_k)=0$.
\end{proof}

\begin{corollary}\label{maincorollary-2} For every $1 \leq k < n < i \leq n+K$, we have that $\tcc(i,0;A_k)=1$ or $\tcc(i,0;L_k)=1$.
\end{corollary}
\begin{proof} This easily follows from Lemma \ref{alconnections}.b and Lemma \ref{notobconnection}.b.\end{proof}

\begin{corollary}\label{maincorollary} For any $1 \leq k < i \leq n$, exactly one of the following holds.
\begin{itemize}
\item[i.] $\tcc(i,1;L_k)=\tcc(i,0;A_k)=1$ and $\tcc(i,1;A_k)=\tcc(i,0;L_k)=0$.
\item[ii.] $\tcc(i,1;L_k)=\tcc(i,0;A_k)=0$ and $\tcc(i,1;A_k)=\tcc(i,0;L_k)=1$.
\end{itemize}
\end{corollary}

\begin{proof} Let $1 \leq k < i \leq n$. We split into two cases.

Suppose that $\tcc(i,1;L_k)=1$. Then, by Lemma \ref{notrianglelemma}.b, we get $\tcc(i,0;L_k)=0$. Now, applying Lemma \ref{alconnections}.a and Lemma \ref{notobconnection}.a, we see that $\tcc(i,0;A_k)=1$. Since $\tcc(i,0;A_k)=1$, it follows from Lemma \ref{notrianglelemma}.b that $\tcc(i,1;A_k)=0$. Therefore, we are in Case (i).

Suppose that $\tcc(i,1;L_k)=0$. Applying Lemma \ref{alconnections}.a and Lemma \ref{notobconnection}.a, we obtain that $\tcc(i,1;A_k)=1$. Then $\tcc(i,0;A_k)=0$ by Lemma \ref{notrianglelemma}.b. Since $\tcc(i,0;A_k)=0$, another application of Lemma \ref{alconnections}.a, and Lemma \ref{notobconnection}.a, gives us that $\tcc(i,0;L_k)=1$. Therefore, we are in Case (ii).
That these cases are mutually exclusive is clear.\end{proof}

With our graph interpretation in mind, this last corollary basically explains why, in the proof of Theorem \ref{mainresult-lower}, the edges in $E_1$ of $\GG_n$ were chosen as they are. Up to certain choices, the edge structure of $\GG_n$ was already mostly determined by the non-existence of a red homogeneous $\omega+n$ and a blue homogeneous $3$. We shall need two more lemmas before we prove the second main theorem and conclude this section.

\begin{lemma}\label{mainlemma} Let $1 \leq i <n$ and $0 \leq j \leq 2$. If we have that $\tcc(m,0;i,j)=1$ for all $n+1 \leq m \leq n+K$, then $\tcc(L_m;i,j)=0$ for all $n+1 \leq m < n+K$.
\end{lemma}

\begin{proof} Assume that $\tcc(m,0;i,j)=1$ whenever $n+1 \leq m \leq n+K$. It follows from Lemma \ref{notrianglelemma}.d that $\tcc(n+K,0;m,0)=0$ for all $n+1 \leq m \leq n+K$.

We claim that $\hcc(m,1,0)=1$ or $\tcc(n+K,0;L_m)=1$ for all $n+1 \leq m \leq n+K$. Suppose towards a contradiction that $\hcc(m,1,0)=\tcc(n+K,0;L_m)=0$ for some $n+1 \leq m \leq n+K$. Then, using the ideas in the proof of Lemma \ref{omegaplusnlemma}, one can construct a red homogeneous closed copy of $\omega+n$ inside the set $$\node{m}{0}{\gamma} \cup \{\omega^2 \cdot n + \omega \cdot m\} \cup \node{n+K}{0}{\gamma}$$
which leads to a contradiction.

Let $n+1 \leq m < n+K$. If $\hcc(m,1,0)=1$, then we must have $\tcc(L_m;i,j)=0$ since, otherwise, having $\tcc(m,0;i,j)=\tcc(L_m;i,j)=\hcc(m,1,0)=1$ would create a blue homogeneous copy of $3$. If $\tcc(n+K,0;L_m)=1$, then we must have $\tcc(L_m;i,j)=0$ since, otherwise, having $\tcc(n+K,0;L_m)=\tcc(n+K,0;i,j)=\tcc(L_m;i,j)=1$ would create a blue homogeneous copy of $3$. (We do not explicitly write the arguments for these claims as they can be done imitating the proof of Lemma \ref{notrianglelemma}.) Therefore, in either case, we have that $\tcc(L_m;i,j)=0$.
\end{proof}

In what follows, for each $1 \leq i < n$, we set $\overline{A_i}=L_i$ and $\overline{L_i}=A_i$. 

\begin{lemma}\label{mainlemma-2} Let $1 \leq k < i \leq n$ and $X \in \{A_k,L_k\}$. If $\tcc(L_i;X)=1$, then $\tcc(L_m;\overline{X})=0$ for all $n+1 \leq m < n+K$.
\end{lemma}

\begin{proof} Assume that $\tcc(L_i;X)=1$. We split into two cases depending on whether or not we have $\hcc(i,2,j)=1$  for some $0 \leq j \leq 1$.

Suppose that $\hcc(i,2,0)=\hcc(i,2,1)=0$, in which case we must have $i=n$ by Lemma \ref{lbconnection}. Now, Lemma \ref{omegaplusnlemma}.b implies that, $\tcc(m,0;L_n)=1$ or  $\tcc(m,0;n, 0)=1$, and $\tcc(m,0;L_n)=1$ or   $\tcc(m,0;n, 1)=1$. But from Lemma \ref{notrianglelemma}.a$'$, we have  $\tcc(m,0;n, 0)=0$ or $\tcc(m,0;n, 1)=0$. So $\tcc(m,0;L_n)=1$ for all $n+1 \leq m \leq n+K$. Consequently, in order to avoid  a blue homogeneous copy of $3$, we must have $\tcc(m,0;X)=0$ for all $n+1 \leq m \leq n+K$. But then, by Lemma \ref{omegaplusnlemma}.b, we must have $\tcc(m,0;\overline{X})=1$ for all $n+1 \leq m \leq n+K$. Consequently, Lemma \ref{mainlemma} implies that $\tcc(L_m;\overline{X})=0$ for all $n+1 \leq m < n+K$.

Now suppose that $\hcc(i,2,0)=1$ or $\hcc(i,2,1)=1$. We would like to remark that we may or may not have $i=n$. Even if $i=n$, let us name the pairs $(i,0)$ and $(i,1)$ by $A_i$ and $B_i$ in such a way that $\tcc(L_i;A_i)=0$ and $\tcc(L_i,B_i)=1$. By Lemma \ref{omegaplusnlemma}.b, we must have that $\tcc(m,0;A_i)=1$ or $\tcc(m,0;L_i)=1$ for all $n+1 \leq m \leq n+K$.

On the other hand, as $\tcc(L_i;X)=1$, we cannot have $\tcc(B_i;X)=1$ in order for there not to be a blue homogeneous $3$. So, by Corollary \ref{maincorollary}, we have $\tcc(A_i;X)=1$. By Lemma \ref{omegaplusnlemma}.b, we have $\tcc(m,0;L_i)=1$ or  $\tcc(m,0;A_i)=1$ for all $n+1 \leq m \leq n+K$. We also have $\tcc(A_i;X)=\tcc(L_i;X)=1$ and hence, the non-existence of a blue homogeneous copy of $3$ now gives us that $\tcc(m,0;X)=0$ for all $n+1 \leq m \leq n+K$. Now applying Lemma \ref{omegaplusnlemma}.b, we must have $\tcc(m,0;\overline{X})=1$ for all $n+1 \leq m \leq n+K$. Subsequently, by Lemma \ref{mainlemma}, we have $\tcc(L_m;\overline{X})=0$ for all $n+1 \leq m < n+K$.
\end{proof}

We are now ready to prove our second main theorem.

\begin{proof}[Proof of Theorem \ref{mainresult-upper}] Let $n \geq 3$ be an integer. Set $K=R(2n-3,3)+1$ and $\gamma=\omega^2 \cdot n + \omega \cdot K+1$. Let $\cc: [\gamma]^2 \rightarrow \{0,1\}$ be a coloring. We wish to show that there exist a red homogeneous closed copy of $\omega+n$ or a blue homogeneous (necessarily closed) copy of $3=\{0,1,2\}$. By the remarks at the end of Section 2, we may assume without loss of generality that $\cc$ is canonical and $\omega$-homogeneous. Assume to the contrary that such homogeneous sets do not exist. Then Lemma \ref{notrianglelemma}-\ref{mainlemma-2} and Corollary \ref{maincorollary-2}-\ref{maincorollary} all hold.

It is clear that $\mathbf{[\boldsymbol{\gamma};A_1]} \cup \{\omega^2\}$ is a red homogeneous copy of $\omega+1$. By assumption, the set $\mathbf{[\boldsymbol{\gamma};A_1]} \cup \{\omega^2 \cdot 1, \omega^2 \cdot 2, \dots, \omega^2 \cdot n\}$, which is a closed copy of $\omega+n$, is not red homogeneous. Consequently, there exists $2 \leq i \leq n$ such that $\tcc(L_i;X)=1$ for some $X \in \{A_1,L_1,L_2,\dots,L_{i-1}\}$, since, otherwise, we would obtain a contradiction by Lemma \ref{omegaplusnlemma2}.

By Corollary \ref{maincorollary}, we have $\tcc(Y,X)=1$ and $\tcc(Y,\overline{X})=0$ where $Y=(i,j)$ for some $0 \leq j \leq 1$. Since there is no blue homogeneous copy of $3$ in $\gamma$, by the definition of  $R(2n-3,3)$, there exist $2n-3$ indices $n+1 \leq m_1 < m_2 < \dots < m_{2n-3} \leq n+K$ such that $$\bigcup_{t=1}^{2n-3} \mathbf{[\boldsymbol{\gamma};L_{m_t}]}$$ is red homogeneous. Now, by the pigeonhole principle, we can find $n-1$ indices $n+1 \leq m'_1 < m'_2 < \dots < m'_{n-1} \leq n+K$ such that
\begin{itemize}
\item $\tcc(L_{m'_t},X)=0$ for all $1 \leq t \leq n-1$ or
\item $\tcc(L_{m'_t},X)=1$ for all $1 \leq t \leq n-1$.
\end{itemize}
Applying Lemma \ref{mainlemma-2}, we see that $\tcc(L_m;\overline{X})=0$ for all $n+1 \leq m < n+K$. So the first case directly contradicts Lemma \ref{omegaplusnlemma2}. Thus the second case holds. But then, since $\tcc(Y,X)=1$, we must have $\tcc(L_{m'_t},Y)=0$ for all $1 \leq t \leq n-1$. Similarly, since $\tcc(L_i,X)=1$, we must have $\tcc(L_{m'_t},L_i)=0$ for all $1 \leq t \leq n-1$. (Otherwise, one can create a blue homogeneous copy of $3$.) On the other hand, since $\tcc(L_i;X)=\tcc(Y;X)=1$, in order for there not to be a blue homogeneous copy of $3$, we must have $\hcc(i,2,j)=0$ where $Y = (i,j)$. Together with the previous observation that $\tcc(L_{m'_t},Y) = \tcc(L_{m'_t},L_i)=0$ for all $1 \leq t \leq n-1$, this contradicts Lemma \ref{omegaplusnlemma2}, which completes the proof.
\end{proof}
\section{Yet another upper bound}

In this section, we shall prove Theorem \ref{mainresult-upper-2} using the ideas that are employed in Section 4. For this reason, we retain all the notation introduced in Section 4.

Analyzing the proof of Theorem \ref{mainresult-upper}, one sees that the whole proof is based on the following phenomenon: If $\mathbf{[\boldsymbol{\gamma};A_1]} \cup \{\omega^2 \cdot 1, \omega^2 \cdot 2, \dots, \omega^2 \cdot n\}$ fails to be red homogeneous, then having sufficiently many $L_i$'s with $n<i$ whose corresponding elements form a red homogeneous set automatically creates a red homogeneous closed copy of $\omega+n$. In that proof, such $L_i$'s were extracted using Ramsey numbers. In the proof of Theorem \ref{mainresult-upper-2}, we shall take another approach to create such $L_i$'s.

\begin{proof}[Proof of Theorem \ref{mainresult-upper-2}] Let $n \geq 3$ be an integer and set $\gamma=\omega^2 \cdot n + \omega \cdot (n^2-4)+1$. Let $\cc: [\gamma]^2 \rightarrow \{0,1\}$ be a canonical $\omega$-homogeneous. Assume to the contrary no red homogeneous closed copy of $\omega+n$ and no blue homogeneous copy of $3$ exist. Then Lemma \ref{notrianglelemma}-\ref{mainlemma-2} and Corollary \ref{maincorollary-2}-\ref{maincorollary} all hold.

As before, there must exist $2 \leq i \leq n$ and $X \in \{A_1,L_1,L_2,\dots,L_{i-1}\}$ such that $\tcc(L_i;X)=1$. Take the least such $i$ and such $X$. We shall now split into two cases depending on whether $i=n$ or $i<n$.

Suppose that $i<n$. By the non-existence of a blue homogeneous $3$, we must have $\tcc(B_i;X)=0$ since $\tcc(L_i;X)=1$. As $\tcc(B_i;X)=0$, we obtain $\tcc(A_i;X)=1$ by Corollary \ref{maincorollary}. On the other hand, we also have by Corollary \ref{maincorollary} that $$\tcc(i+1,j;L_i)=\tcc(i+1,1-j;A_i)=1$$ for some $0 \leq j \leq 1$. Thus $$\tcc(i+1,j;L_i)=\tcc(L_i;X)=\tcc(i+1,1-j;A_i)=\tcc(A_i;X)=1$$
But then, it follows that $$\tcc(i+1,j;X)=\tcc(i+1,1-j,X)=0$$ because, otherwise, one can construct a blue homogeneous copy of $3$. However, the last equality contradicts Corollary \ref{maincorollary}.

Now suppose that $i=n$. Set $W_X=\{L_j: \tcc(L_j;X)=1,\ n+1 \leq j \leq n+(n^2-4)\}$ and consider the sets
\begin{align*}
W&=\{L_j: \tcc(L_j;X)=0,\ n+1 \leq j \leq n+(n^2-4)\}\\
W_{n-1}&=\{L_j \in W: \tcc(L_j;A_{n-1})=1 \text{ or } \tcc(L_j;L_{n-1})=1\}\\
W_{n-2}&=\{L_j \in W: L_j \notin W_{n-1},\ \tcc(L_j;A_{n-2})=1 \text{ or } \tcc(L_j;L_{n-2})=1\}\\
\dots&\\
W_{k}&=\left\{L_j \in W: L_j \notin \left(\bigcup_{m=k+1}^{n-1} W_{m} \right),\ \tcc(L_j;A_{k})=1 \text{ or } \tcc(L_j;L_k)=1\right\}\\
\dots&\\
W_{1}&=\left\{L_j \in W: L_i \notin \left(\bigcup_{m=2}^{n-1} W_{m} \right),\ \tcc(L_j;A_{1})=1 \text{ or } \tcc(L_j;L_1)=1\right\}
\end{align*}
Recall that $i$ was chosen to be the least integer with its property. It follows that the set $\mathbf{[\boldsymbol{\gamma};A_1]} \cup \{\omega^2 \cdot 1, \omega^2 \cdot 2, \dots, \omega^2 \cdot (n-1)\}$ is a red homogeneous closed copy of $\omega+(n-1)$. Consequently, in order for there not be a red homogeneous closed copy of $\omega+n$, we must have that, for every $n+1 \leq j \leq n+(n^2-4)$, there exists $Y \in \{A_1,L_1,L_2,\dots,L_{n-1}\}$ such that $\tcc(L_j;Y)=1$. It follows that
\[ W_X \cup \left(\bigcup_{m=1}^{n-1} W_m\right) = \{L_j: n+1 \leq j \leq n+(n^2-4)\}\]
Next will be made some important observations.
\begin{itemize}
\item Let $1 \leq k \leq n-2$. Suppose that $W_k$ has at least $(2k+1)$ elements. Then, by the pigeonhole principle, we must have that $\tcc(L_j;A_{k})=1$ for at least $k+1$ of these elements or $\tcc(L_j;L_{k})=1$ for at least $k+1$ of these elements. In either case, since there is no blue homogeneous copy of $3$, there exists $k+1$ elements $L_{m_1},L_{m_2},\dots,L_{m_{k+1}}$ in $W_k$ such that $\tcc(L_{m_t};L_{m_{t'}})=0$ whenever $1 \leq t \neq t' \leq k+1$. Recall that $\tcc(L_{j};L_{j'})=0$ for all $1 \leq j \neq j' <n$. Moreover, because $L_{m_t} \in W_k$, we have that $\tcc(L_{m_t};A_{k+1})=0$ and that $\tcc(L_{m_t};L_{j})=0$ for all $1 \leq t \leq k+1$ and all $k+1 \leq j \leq n-1$. This contradicts Lemma \ref{omegaplusnlemma2}.
\item Suppose that $W_{n-1}$ has at least $(2n-3)$ elements. Then, as above, there are $L_{m_1},L_{m_2},\dots,L_{m_{n-1}}$ in $W_{n-1}$ such that $\tcc(L_{m_t};L_{m_{t'}})=0$ whenever $1 \leq t \neq t' \leq n-1$. Since $W_{n-1} \subseteq W$, we have that $\tcc(L_{m_t};X)=0$ for all $1 \leq t \leq n-1$. Moreover, by Lemma \ref{mainlemma-2}, we have $\tcc(L_{m_t};\overline{X})=0$ for all $1 \leq t \leq n-1$. This contradicts Lemma \ref{omegaplusnlemma2}.
\item Suppose that $W_X$ has at least $n-1$ elements, say, $L_{m_1},L_{m_2},\dots,L_{m_{n-1}}$. Then, since $\tcc(L_{m_t};X)=\tcc(L_n;X)=1$ for all $1 \leq t \leq n-1$, we must have that $\tcc(L_{m_t};L_{m_{t'}})=\tcc(L_{m_t};L_n)=0$ whenever $1 \leq t \neq t' \leq n-1$ because there is no blue homogeneous $3$. By Corollary \ref{maincorollary}, there exists $0 \leq \ell \leq 1$ such that $\tcc(n,\ell;X)=1$. Again, by the non-existence of a blue homogeneous $3$, we obtain that $\tcc(L_{m_t};n,\ell)=\hcc(n,2,\ell)=0$ for all $1 \leq t \leq n-1$. Together with previous observations, this contradicts Lemma \ref{omegaplusnlemma2}.
\end{itemize}
Therefore, $W_k$ has at most $2k$ elements for all $1 \leq k \leq n-2$, $W_{n-1}$ has at most $2n-4$ elements and $W_X$ has at most $n-2$ elements. This means that $W_X \cup \left(\cup_{m=1}^{n-1} W_m\right)$ can have at most
\[ \left(\sum_{k=1}^{n-2} 2k \right) + (2n-4) + (n-2) =(n-2)(n-1)+(2n-4) + (n-2)=n^2-4\]
elements. We will now argue that it can indeed have at most $n^2-5$ elements. Recall that $X \in \{A_1,L_1,L_2,\dots,L_{n-1} \}$. Consequently, the corresponding $W_k$ is either empty or equal to $\{L_{n^2-4}\}$, because, by Lemma \ref{mainlemma-2}, we have $\tcc(L_m;\overline{X})=0$ for all $n+1 \leq m < n+(n^2-4)$. Thus, the corresponding $W_k$ can have at most one element instead of at most $2k$ elements. So the union $W_X \cup \left(\cup_{m=1}^{n-1} W_m\right)$ can indeed have at most $n^2-4-(2k-1)$ elements.

Letting $k=1$, we obtain that $W_X \cup \left(\cup_{m=1}^{n-1} W_m\right)$ can actually have at most $n^2-5$ elements. However, we had
$W_X \cup \bigcup_{m=1}^{n-1} W_m = \{L_j: n+1 \leq j \leq n+(n^2-4) \}$ which now gives a contradiction since the right-hand side has $n^2-4$ elements. This completes the proof.\end{proof}

\bibliography{references}{}

\providecommand{\bysame}{\leavevmode\hbox to3em{\hrulefill}\thinspace}
\providecommand{\MR}{\relax\ifhmode\unskip\space\fi MR }
\providecommand{\MRhref}[2]{%
  \href{http://www.ams.org/mathscinet-getitem?mr=#1}{#2}
}
\providecommand{\href}[2]{#2}
\begin{thebibliography}{Kim95}

\bibitem[Bau86]{Baumgartner86}
James~E. Baumgartner, \emph{Partition relations for countable topological
  spaces}, J. Combin. Theory Ser. A \textbf{43} (1986), no.~2, 178--195.
  \MR{867644}

\bibitem[CH17]{CaicedoHilton17}
Andr\'{e}s~Eduardo Caicedo and Jacob Hilton, \emph{Topological {R}amsey numbers
  and countable ordinals}, Foundations of mathematics, Contemp. Math., vol.
  690, Amer. Math. Soc., Providence, RI, 2017, pp.~87--120. \MR{3656308}

\bibitem[ER53]{ErdosRado53}
P.~Erd\"{o}s and R.~Rado, \emph{A problem on ordered sets}, J. London Math.
  Soc. \textbf{28} (1953), 426--438. \MR{58687}

\bibitem[ER56]{ErdosRado56}
\bysame, \emph{A partition calculus in set theory}, Bull. Amer. Math. Soc.
  \textbf{62} (1956), 427--489. \MR{81864}

\bibitem[Kim95]{Kim95}
Jeong~Han Kim, \emph{The {R}amsey number {$R(3,t)$} has order of magnitude
  {$t^2/\log t$}}, Random Structures Algorithms \textbf{7} (1995), no.~3,
  173--207. \MR{1369063}

\bibitem[Mer19]{Mermelstein19}
Omer Mermelstein, \emph{Calculating the closed ordinal {R}amsey number
  {$R^{cl}(\omega\cdot 2, 3)$}}, Israel J. Math. \textbf{230} (2019), no.~1,
  387--407. \MR{3941152}

\end{thebibliography}
\bibliographystyle{amsalpha}

\end{document}